\documentclass[11pt,reqno]{amsart}
\usepackage{amscd}

\numberwithin{equation}{section}
\usepackage{times}
\usepackage{amsmath,amsfonts,amstext,amssymb,amsbsy,amsopn,amsthm,eucal}
\usepackage{mathrsfs}
\usepackage{txfonts}
\usepackage{dsfont}
\usepackage{graphicx}   
\usepackage{hyperref}
\usepackage{accents}
\usepackage{enumerate}
\usepackage{xcolor}

\newtheorem{theorem}[equation]{Theorem}

\newtheorem{proposition}[equation]{Proposition}
\newtheorem{lemma}[equation]{Lemma}
\newtheorem{corollary}[equation]{Corollary}
\theoremstyle{definition}
\newtheorem{definition}[equation]{Definition}

\theoremstyle{remark}
\newtheorem{remark}[equation]{Remark}
\theoremstyle{remark}

\theoremstyle{remark}

\theoremstyle{remark}
\theoremstyle{remark}

\begin{document}
\title{On the deformation of inversive distance circle packings, II}
\author{Huabin Ge, Wenshuai Jiang}
\date{\today}

\maketitle

\begin{abstract}
We show that the results in \cite{Ge-Jiang1} are still true in hyperbolic background geometry setting, that is, the solution to Chow-Luo's combinatorial Ricci flow can always be extended to a solution that exists for all time, furthermore, the extended solution converges exponentially fast if and only if there exists a metric with zero curvature. We also give some results about the range of discrete Gaussian curvatures, which generalize Andreev-Thurston's theorem to some extent.
\end{abstract}

\section{Introduction}\label{Introduction}
\subsection{Background}
We continue our study about the deformation of inversive distance circle patterns on a surface $M$ with triangulation $\mathcal{T}$. If the triangulated surface is obtained by taking a finite collection of Euclidean triangles in $\mathds{E}^2$ and identifying their edges in pairs by isometries, we call it in Euclidean background geometry. While if the triangulated surface is obtained by taking a finite collection of hyperbolic triangles in $\mathds{H}^2$ and identifying their edges in pairs by isometries, we shall call it in hyperbolic background geometry.

Consider a closed surface $M$ with a triangulation $\mathcal{T}=\{V,E,F\}$, where $V,E,F$ represent the sets of vertices, edges and faces respectively. Thurston \cite{T1} once introduced a metric structure on $(M, \mathcal{T})$ by circle patterns. The idea is taking the triangulation as the nerve of a circle pattern, while the length structure on $(M, \mathcal{T})$ can be constructed from radii of circles and intersection angles between circles in the pattern. More precisely, let $\Phi_{ij}\in [0,\frac{\pi}{2}]$ be intersection angles between two circles $c_i$, $c_j$ for each nerve $\{ij\}\in E$, and let $r_i\in(0,+\infty)$ be the radius of circle $c_i$ for each vertex $i\in V$, see Figure \ref{figure-circle-packing}.
One can use the cosine law in $\mathds{E}^2$ or $\mathds{H}^2$, to equip each edge $\{ij\}\in E$ with a length $l_{ij}=\sqrt{r_i^2+r_j^2+2r_ir_j\cos \Phi_{ij}}$ in Euclidean background geometry and $l_{ij}=\cosh^{-1}(\cosh r_i\cosh r_j+\sinh r_i\sinh r_j\cos \Phi_{ij})$ in hyperbolic background geometry. Thurston proved that (\cite{T1}, Lemma 13.7.2), in both two background geometry and for each face $\{ijk\}\in F$, the three lengths $l_{ij}, l_{jk}, l_{ik}$ satisfy the triangle inequalities. This makes each triangle in $F$ isometric to a triangle in $\mathds{E}^2$ in Euclidean background geometry and a triangle in $\mathds{H}^2$ in hyperbolic background geometry. Furthermore, the triangulated surface $(M, \mathcal{T})$ could be taken as gluing many Euclidean (hyperbolic) triangles coherently in Euclidean (hyperbolic) background geometry. Suppose $\theta_i^{jk}$ is the inner angle of the triangle $\{ijk\}$ at the vertex $i$, the classical well-known discrete Gaussian curvature at each vertex $i$ is
\begin{equation}\label{classical Gauss curv}
K_i=2\pi-\sum_{\{ijk\} \in F}\theta_i^{jk},
\end{equation}
and the discrete curvature $K_i$ satisfies the following discrete version of Gauss-Bonnet formula \cite{CL1}:
\begin{equation}\label{Gauss-Bonnet}
\sum_{i=1}^NK_i=2\pi \chi(M)+\lambda Area(M).
\end{equation}
Here $\lambda=0$ in Euclidean background geometry and $\lambda=1$ in hyperbolic background geometry. Andreev-Thurston observed the rigidity property of circle patterns, that is, a circle packing metric is uniquely determined by its curvatures up to Euclidean similarity (hyperbolic isometry) in Euclidean (hyperbolic) background geometry, see \cite{Andreev2, Andreev1, T1, Marden-Rodin, Colindev} for a proof. In the pioneered work of \cite{CL1}, Chow and Luo first established an intrinsic connection between Thurston's circle patterns and surface Ricci flow. They introduced the so called combinatorial Ricci flows,
\begin{equation}\label{Def-ChowLuo-flow-euclid}
{r_i}'(t)=-K_ir_i
\end{equation}
in Euclidean background geometry and
\begin{equation}\label{Def-ChowLuo-flow-hyper}
{r_i}'(t)=-K_i\sinh r_i
\end{equation}
in hyperbolic background geometry. These discrete flows are analog of Hamilton's Ricci flow in the combinatorial setting. Chow and Luo proved that the solutions to combinatorial Ricci flows exist for all time. Moreover, the solutions converges (to a circle pattern with constant curvature)
if and only if there exists a metric of constant curvature. As a consequence, they obtained a new proof of Thurston's existence of circle packing theorem and a new algorithm to deform circle patterns.
\begin{figure}\label{figure-circle-packing}
\begin{minipage}[t]{0.48\linewidth}
\centering
\includegraphics[width=0.75\textwidth]{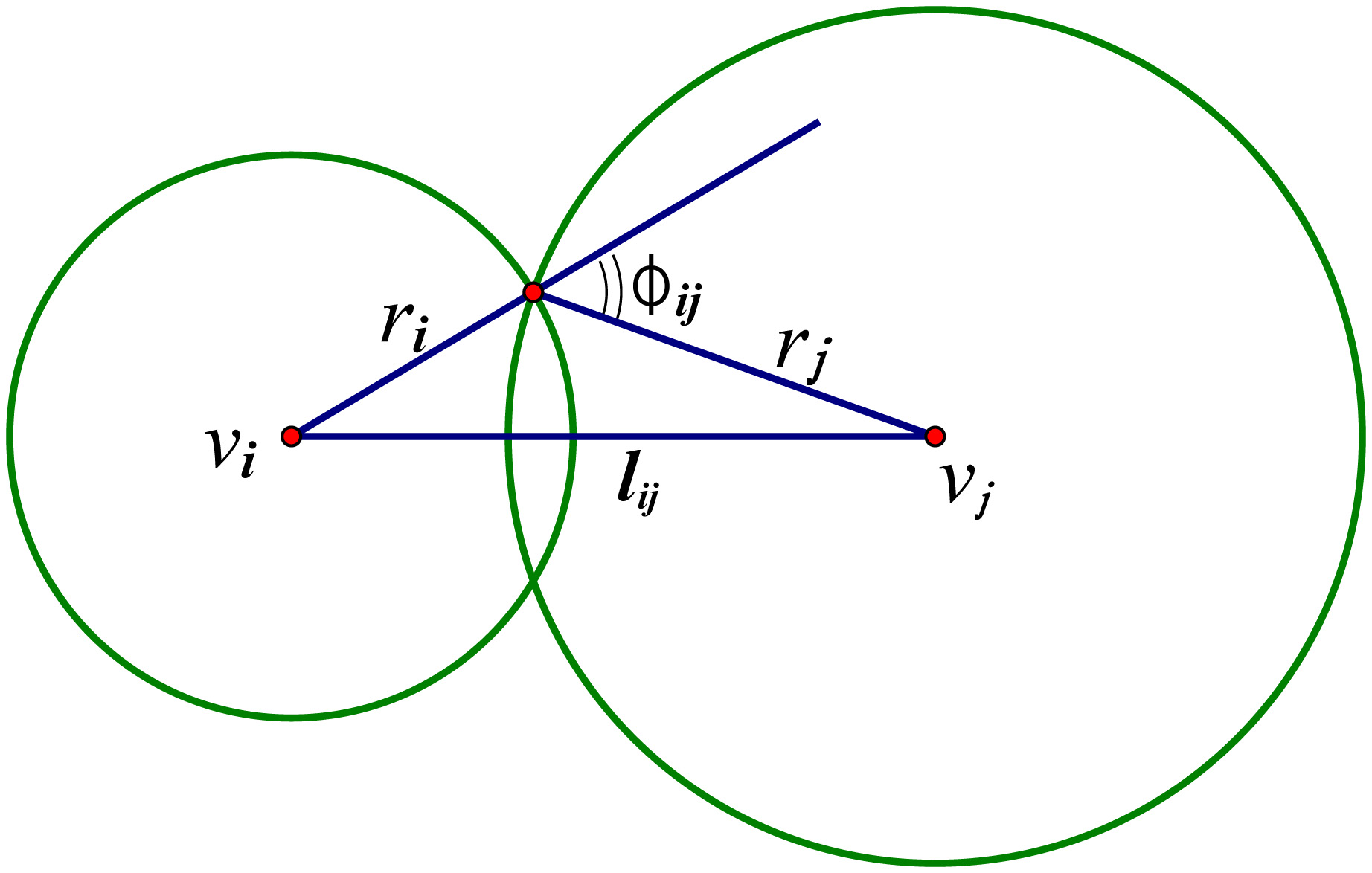}
\caption{circle packing metric}
\end{minipage}
\end{figure}

\subsection{Inversive distance circle packing metric}\label{ICP-metric}
However, Andreev and Thurston's circle patterns require adjacent circles intersect with each other, which is too restrictive. Hence Bowers and Stephenson \cite{Bowers-Stephenson} introduced inversive distance circle packing, which allow adjacent circles to be disjoint and measure their relative positions by the inversive distance. Consider two circles $c_1$, $c_2$ with radii $r_1$, $r_2$ respectively, and assume that $c_1$ does not contain $c_2$ and vice versa. If the distance between their center is $l_{12}$, then the inversive distance between $c_1$, $c_2$ is given by formula
\begin{equation}\label{formula-inver-dist-euclid}
I(c_1, c_2)=\cfrac{l_{12}^2-r_1^2-r_2^2}{2r_1r_2}.
\end{equation}
in Euclidean background geometry and formula
\begin{equation}\label{formula-inver-dist-hyper}
I(c_1, c_2)=\cfrac{\cosh l_{12}-\cosh r_1\cosh r_2}{\sinh r_1\sinh r_2}.
\end{equation}
in hyperbolic background geometry. These two formulas are related by a stereographic projection (\cite{Guoren}). The notion of inversive distance generalizes the notion of intersection angle of two circles. Actually, note that $l_{12}> |r_1-r_2|$, we have $-1<I(c_1, c_2)< +\infty$. The inverse distance $I(c_1, c_2)$ describes the relative positions of $c_1$ and $c_2$. When $I(c_1, c_2)\in (-1, 0)$, the circles $c_1$, $c_2$ intersect with an intersection angle $\arccos I(c_1, c_2)\in (\frac{\pi}{2}, \pi)$. When $I(c_1, c_2)\in [0, 1]$, the circles $c_1$, $c_2$ intersect with an intersection angle $\arccos I(c_1, c_2)\in [0, \frac{\pi}{2}]$. When $I(c_1, c_2)\in (1, +\infty)$, the circles $c_1$, $c_2$ are separated.

Now we reformulate Bowers and Stephenson's construction of an inversive distance circle packing with prescribed inversive distance $I$ on triangulated surface $(M, \mathcal{T})$. Consider $I$ as a function defined on all edges; that is $I: E\rightarrow (-1, +\infty)$, and we call $I$ the inversive distance. For every given radius vector $r\in \mathds{R}^N_{>0}$, we equip each edge $\{ij\}\in E$ with a length
\begin{equation}\label{Def-edge-length}
l_{ij}=\sqrt{r_i^2+r_j^2+2r_ir_jI_{ij}}
\end{equation}
in Euclidean background geometry and length
\begin{equation}\label{formula-inver-dist-hyper}
l_{ij}=\cosh^{-1}(\cosh r_i\cosh r_j+I_{ij}\sinh r_i\sinh r_j)
\end{equation}
in hyperbolic background geometry. We should be careful that for $\{ijk\}\in F$, $l_{ij}, l_{jk}, l_{ik}$ may not satisfy the triangle inequalities any more in the inversive distance setting, which is quite different from Thurston's observation, i.e., Lemma 13.7.2 in \cite{T1}. Denote
\begin{equation}
\Omega=\Big\{r\in \mathds{R}^N_{>0}\;\big|\;l_{ij}+l_{jk}>l_{ik},\;l_{ij}+l_{ik}>l_{jk},\;l_{ik}+l_{jk}>l_{ij}, \;\forall \;\{ijk\}\in F\Big\}
\end{equation}
as the space of all possible inversive distance circle packing metrics.

\subsection{Our main results}
Bowers-Stephenson's relaxation of intersection condition is very useful for practical applications, especially in medical imaging and computer graphics fields, see Bowers-Hurdal \cite{Bowers-Hurdal} and Hurdal-Stephenson \cite{Hurdal-Stephenson} for example. Based on extensive numerical evidences, Bowers and Stephenson \cite{Bowers-Stephenson} conjectured the rigidity and convergence of inversive distance circle packings. Hence it is quite necessary to generalize Andreev-Thurston's ridigity property to inversive distance circle pattern setting. Guo \cite{Guoren} first proved that Bowers-Stephenson＊s conjecture of rigidity is locally true by complicated computations. Luo \cite{Luo1} proved Bowers-Stephenson＊s conjecture of rigidity eventually. Their results say that a inversive distance circle packing metric is uniquely determined by its curvatures (up to Euclidean similarity or hyperbolic isometry). More precisely, the curvature map $K:\Omega\rightarrow \mathds{R}^N$, $r\mapsto K$ is injective in hyperbolic background geometry and is injective up to scaling in Euclidean background geometry. It is also necessary to generalize Chow-Luo's combinatorial Ricci flow to inversive distance circle pattern setting. In fact, in Euclidean background geometry, we \cite{Ge-Jiang1} had extended the solution to Chow-Luo's combinatorial Ricci flow so as it exists for all time and converges exponentially fast to a circle pattern with constant cone angles. In this paper, we shall do the same thing in hyperbolic background geometry. Combine Proposition \ref{Prop-maxm-exists-time} and Theorem \ref{thm-long-term-existence-extended}, we get
\begin{theorem}[Extendable flow]
Given a triangulated surface $(M, \mathcal{T})$ with inversive distance $I\geq 0$ in hyperbolic background geometry. For any initial value $r(0)\in \Omega$, the solution $\{r(t)|0\leq t<T\}\subset\Omega$ to flow (\ref{Def-ChowLuo-flow-hyper}) uniquely exists on a maximal time interval $t\in[0,T)$ with $0<T\leq +\infty$. Furthermore, we can always extend the solution to a new solution $\{r(t)|t\in[0,+\infty)\}$ whenever $T<+\infty$.
\end{theorem}
In the proof, we first follow the idea from Bobenko-Pincall-Springborn \cite{Bobenko} and Luo \cite{Luo1} to extend the definition of curvature $K$ so as it defined on the whole space $\mathds{R}^N_{>0}$. Then we consider the extension of flow equation (\ref{Def-ChowLuo-flow-hyper}). Using the tricks developed by Ge-Xu \cite{Ge-Xu2, Ge-Xu4} and Ge-Jiang \cite{Ge-Jiang1}, we finally obtain the extended solution $\{r(t)|0\leq t< +\infty\}$. By deep analysis into Guo-Luo's combinatorial Ricci potential, we generalize Andreev-Thurston and Guo-Luo's rigidity of inversive distance circle patterns to
\begin{theorem}\label{corollary-unique-K-bar}
Any curvature $\bar{K}\in K(\Omega)$ is realized by an unique metric $\bar{r}$ in the extended space $\mathds{R}^N_{>0}$. As a consequence, if there exists a metric of zero curvature in $\Omega$, then it exists uniquely in $\mathds{R}^N_{>0}$.
\end{theorem}
Note the ``consequence" part of above theorem is restated as Theorem \ref{thm-unique-zero-curvature}, while the general part of above theorem is restated as Theorem \ref{thm-unique-admissible-curvature}. As to the long tern convergence behavior of flow (\ref{Def-ChowLuo-flow-hyper}), we have
\begin{theorem}[Convergence of extended flow]\label{Thm-converge-introduction}
Given a triangulated surface $(M, \mathcal{T})$ with inversive distance $I\geq 0$ in hyperbolic background geometry. Then there exists a metric of zero curvature $r^*\in\Omega$ on $M$ if and only if $r(t)$ can always be extended to a convergent solution. Furthermore, if the extended solution converges, then it converges exponentially fast to $r^*$ as $t\rightarrow+\infty$.
\end{theorem}


The paper is organized as follows. In section \ref{section-basic}, we study some basic properties of flow \ref{Def-ChowLuo-flow-hyper}), it contains uniqueness, uniform lower or upper bound, and local convergence behavior. In section \ref{section-extend-flow}, we study the extended flow. Using Luo's extension of locally convex functional, we then prove the extended combinatorial Ricci potential is proper and get the convergence results related to the extended flow. We also introduce a prescribed flow to deform inversive distance circle pattern to a pattern with admissible curvatures. In section \ref{section-degenerate-pattern}, we study the degeneration of inversive distance circle patterns. We also give a necessary condition, expressed by a system of combinatorial and topological obstructions, for the existance of zero curvature metric. We further analysis the effect on the existence of zero curvature metrics due to a deficiency of maximum principle. In section \ref{section-open-question}, we state some open questions related to this subject.

\section{Basic properties of flow (\ref{Def-ChowLuo-flow-hyper})}\label{section-basic}
In this section, we state some basic properties of flow (\ref{Def-ChowLuo-flow-hyper}). Set $u_i=\ln \tanh \frac{r_i}{2}$, then
$du_i=\frac{1}{\sinh r_i}dr_i$. Note $r\mapsto u$ is a coordinate change, sending $r\in\mathds{R}^N_{>0}$ homeomorphically to $u\in\mathds{R}^N_{<0}$. Using this coordinate change, we can write flow (\ref{Def-ChowLuo-flow-hyper}) to an autonomous ODE system
\begin{equation}\label{Def-hyper-flow-u}
\begin{cases}
{u_i}'(t)=-K_i\\
\,u(0)\in \ln\tanh\frac{\Omega}{2}
\end{cases}.
\end{equation}
where $\ln\tanh\frac{\Omega}{2}\triangleq\left\{\big(\ln \tanh \frac{r_1}{2},\cdots,\ln\tanh\frac{r_N}{2}\big)|r\in \Omega\right\}$
is the space of all meaningful inversive circle packing metrics expressed within $u$-coordinate. Since it is no difference between flow (\ref{Def-ChowLuo-flow-hyper}) and flow (\ref{Def-hyper-flow-u}), we will not distinguish them in the following.

Note that, in $\ln\tanh \frac{\Omega}{2}$, $K_i$ as a function of $u=(u_1,\cdots,u_N)^T$ is smooth and hence locally Lipschitz continuous. By Picard theorem in classical ODE theory, flow (\ref{Def-hyper-flow-u}) has a unique solution $u(t)$, $t\in[0, \epsilon)$ for some $\epsilon>0$. As a consequence, we have
\begin{proposition}\label{Prop-maxm-exists-time}
Given a triangulated surface $(M, \mathcal{T})$ with inversive distance $I\geq 0$ in hyperbolic background geometry. For any initial value $r(0)\in \Omega$, the solution $\{r(t)|0\leq t<T\}\subset\Omega$ to flow (\ref{Def-ChowLuo-flow-hyper}) uniquely exists on a maximal time interval $t\in[0,T)$ with $0<T\leq +\infty$.
\end{proposition}

Let $\{r(t)|0\leq t<T\}$ be the unique solution to flow (\ref{Def-ChowLuo-flow-hyper}) on a right maximal time interval $[0, T)$ with $0<T\leq +\infty$. Intuitively, $r(t)$ touches the boundary of $\Omega$ as $t\uparrow T$. Roughly speaking, the boundary of $\Omega$ can be classified into three types. The first type is ``$0$ boundary", $r(t)$ touches the ``$0$ boundary" means that there exists a sequence of time $t_n\uparrow T$ and a vertex $i\in V$ so that $r_i(t_n)\rightarrow 0$. The second type is ``$+\infty$ boundary", $r(t)$ touches the ``$+\infty$ boundary" means that there exists $t_n\uparrow T$ and $i\in V$ so that $r_i(t_n)\rightarrow +\infty$. The last type is ``triangle inequality invalid boundary", for this case, there exists $t_n\uparrow T$ and a triangle $\{ijk\}\in F$, such that the triangle inequality in triangle $\{ijk\}\in F$ do not hold any more as $n\rightarrow +\infty$. At first glance the limit behavior of $r(t)$ as $t\uparrow T$ may be mixed of the three types and may be very complicated. We prove that in any finite time interval, $r(t)$ never touches the ``$0$ boundary" and ``$+\infty$ boundary".

\begin{proposition}\label{Prop-positive-low-bound}
Given a triangulated surface $(M, \mathcal{T})$ with inversive distance $I\geq 0$ in hyperbolic background geometry. Let $r(t)$ be the solution to flow (\ref{Def-ChowLuo-flow-hyper}) on a right maximal time interval $[0, T)$ with $0<T\leq +\infty$, then every $r_i(t)$ has a positive lower bound in any finite time interval $[0, a)\cap[0,T)$ with $a<+\infty$.
\end{proposition}
\begin{proof}
Note that, all $|K_i|$ are uniformly bounded by a constant $c>0$, which depends only on the triangulation. Denote $c_i=\tanh\frac{r_i(0)}{2}$, obviously $c_i\in(0,1)$, then
$$r_i(t)\geq\ln \frac{1+c_ie^{-ct}}{1-c_ie^{-ct}}>0,$$
which implies the conclusion above.
\end{proof}

\begin{lemma}\label{Lemma-edge-big-then-angle-zero}
For any $\epsilon>0$, there exists a number $l$ so that when $r_i>l$, the inner angle $\theta_i$ in the hyperbolic triangle $\triangle v_iv_jv_k$ is smaller than $\epsilon$.
\end{lemma}

Chow and Luo first state above lemma, see Lemma 3.5 in \cite{CL1}. For an elementary proof, see \cite{Ge-Xu4}, Lemma 3.2.

\begin{proposition}\label{Prop-uniform-up-bound}
Given a triangulated surface $(M, \mathcal{T})$ with inversive distance $I\geq 0$ in hyperbolic background geometry. Let $\{r(t)|0\leq t<T\}$ be the unique solution to flow (\ref{Def-ChowLuo-flow-hyper}) on a right maximal time interval $[0, T)$ with $0<T\leq +\infty$. Then all $r_i(t)$ are uniformly bounded above on $t\in[0,T)$.
\end{proposition}
\begin{proof}
We use Ge-Xu's tricks first developed in \cite{Ge-Xu4}. If the proposition is not true, then there exists at least one vertex $i\in V$, such that $\overline{\lim\limits_{t\uparrow T}}~r_i(t)=+\infty$.
For this vertex $i$, using Lemma \ref{Lemma-edge-big-then-angle-zero}, we can choose $l>0$ large enough so that, whenever $r_i>l$,
the inner angle $\theta_i^{jk}$ is smaller than $\pi/d_i$, where $d_i$ is the valence (or say degree) at $i$. Thus $K_i>\pi$. Choose $t_0\in(0,T)$ such that $r_i(t_0)>l$, this can be done since $\overline{\lim\limits_{t\uparrow T}}~r_i(t)=+\infty$. Denote $a=\inf\{\,t<t_0\,|\,r_i(t)>l\,\}$, then $r_i(a)=l$. Let's look at what happens to flow equation (\ref{Def-ChowLuo-flow-hyper}) in the time interval $[a, t_0]$. Note that $r'_i(t)=-K_i\sinh r_i<0$ for $a\leq t\leq t_0$, hence $r_i(t)\leq r_i(a)=l$, which contradicts with $r_i(t_0)>l$.
\end{proof}
\begin{remark}
Chow-Luo proved this result for Andreev-Thurston's circle pattern case by using a maximum principle. However, for inversive distance circle pattern setting, the maximum principle doesn't work, see subsection \ref{section-max-principle} in the following.
\end{remark}

\begin{theorem}[Local stable]\label{Thm-0curv-metric-exist-imply converg}
Given $(M, \mathcal{T})$ with inversive distance $I\geq 0$ in hyperbolic background geometry.
\begin{enumerate}
  \item If the solution $\{r(t)\}\subset\Omega$ to flow (\ref{Def-ChowLuo-flow-hyper}) converges to a metric $r^*\in\Omega$, then $r^*$ is a zero curvature metric. As a consequence, there exists a zero curvature metric on $(M, \mathcal{T})$.
  \item Conversely, assuming there exists a zero curvature metric $r^*\in\Omega$. Then the solution $r(t)$ to flow (\ref{Def-ChowLuo-flow-hyper}) exists for all time $t\geq0$ and converges exponentially fast to $r^*$ if $r(0)$ is close enough to $r^*$.
\end{enumerate}
\end{theorem}

\begin{proof}
We say the solution $r(t)$ to flow (\ref{Def-ChowLuo-flow-hyper}) converges, if $r(t)$ exists for all time $t\in[0,+\infty)$, and $r(t)$ converges in the Euclidean space topology to some $r^*\in\Omega$ as time $t$ tends to $+\infty$. According to the classical ordinary differential equation theory, $r^*$ is a zero point of $-K_i\sinh r_i$. Hence $K_i(r^*)=0$ for each $i$, and $r^*$ is the unique zero curvature metric. Conversely, assume $r^*\in\Omega$ is a zero curvature metric. Write flow (\ref{Def-hyper-flow-u}) in matrix form as $\dot{u}=-K$, then differentiate $K$ at $u^*$, we get
$$D_u(-K)\big|_{u^*}=-\frac{\partial(K_{1},\cdots,K_{N})}{\partial(u_{1},\cdots,u_{N})}=-L.$$
By Guo Ren's result that $L$ is positive definite at each $r\in\Omega$ (see the proof of Corollary in \cite{Guoren}), and Lyapunov stability theorem in ODE theory, $u^*$
is an asymptotically stable point of the flow (\ref{Def-ChowLuo-flow-hyper}), this implies the conclusion above.
\end{proof}

\section{extended combinatorial Ricci flow}\label{section-extend-flow}
Let $\{r(t)|0\leq t<T\}$ be the maximal unique solution to flow (\ref{Def-ChowLuo-flow-hyper}). If $T< +\infty$, then by Proposition \ref{Prop-positive-low-bound} and \ref{Prop-uniform-up-bound}, there exist constants $c, C$ so that $0<c\leq r_i(t)\leq C$ for each $i\in V$ and all $t\in[0,T)$. As a consequence, all edges $l_{ij}(t)$ remain positive for $t\in[0,T)$. Thus $r(t)$ touches only the ``triangle inequality invalid boundary" of $\Omega$ as $t\uparrow T$, that is, there exist $t_n\uparrow T$ and some triangles in $F$, such that the triangle inequality in these triangles do not hold any more as $n\rightarrow +\infty$. However, for these degenerate triangles, we can give reasonable meaning of inner angles, therefore we can extend the definition of discrete Gaussian curvatures to triangulations gluing by both ordinary and degenerate triangles. Using this extension of curvature, we extend flow (\ref{Def-ChowLuo-flow-hyper}) naturally. It's very interesting that the degenerate triangles can be deformed again along the extended flow. In this meaning, we finally extend the solution $r(t)$ so that it exists for all time $t\in[0,+\infty)$.

\subsection{Generalized triangle and generalized curvature}
We follow \cite{Bobenko,Luo1,Luo2} to generalize the definition of inner angles of triangles and \cite{Ge-Jiang1} to generalize the definition of discrete Gaussian curvatures. The key point is to extend the definition of inner angles, and as to the extension of $K_i$, it is just a combinatorial process.

A generalized hyperbolic (or Euclidean) triangle $\triangle v_1v_2v_3$ is a topological triangle of vertices $v_1, v_2, v_3$ so that each edge is assigned a positive number, called edge length. Let $x_i$ be the assigned length of the edge $v_jv_k$ where $\{i, j, k\}$=$\{1, 2, 3\}$. The inner angle $\tilde{\theta}_i$=$\tilde{\theta}_i(x_1, x_2, x_3)$ at the vertex $v_i$ is defined as follows. If $x_1, x_2, x_3$ satisfy the triangle inequalities, then $\tilde{\theta}_i$ is the inner angle of the hyperbolic (or Euclidean) triangle of edge lengths $x_1, x_2, x_3$ opposite to the edge of length $x_i$; if $x_i\geq x_j+x_k$, then $\tilde{\theta}_i=\pi$, $\tilde{\theta}_j=\tilde{\theta}_k=0$. Denote
\begin{equation}
\Xi\triangleq\left\{(x_1,x_2,x_3)\in\mathds{R}^3_{>0}\,\big|\,x_i+x_j>x_k,\, \{i,j,k\}=\{1,2,3\}\right\}
\end{equation}
Obviously, the inner angles $\theta_i$, $\theta_j$ and $\theta_k$ which are originally defined on $\Xi$, are now extend to $\tilde{\theta}_i$, $\tilde{\theta}_j$ and $\tilde{\theta}_k$ which are defined on $\mathds{R}^3_{>0}$. Luo (Lemma 2.2,\cite{Luo1}) already showed that the extension is continuously. We can express the extended inner angles $\tilde{\theta}_i:\mathds{R}^3_{>0}\rightarrow [0,\pi]$, $i\in \{i,j,k\}$ more clearly and directly. In fact, recall the auxiliary function $\Lambda(x)$ we introduced in \cite{Ge-Jiang1}
\begin{equation}\label{def-big-lamda}
\Lambda(x)=
\begin{cases}
\;\;\;\;\;\pi, & \text{$x \leq -1$.}\\
\arccos x, & \text{$-1\leq x \leq 1$.}\\
\;\;\;\;\;0,& \text{$x\geq 1$.}
\end{cases}
\end{equation}
$\Lambda$ is continuous on $\mathds{R}$ and $\Lambda(-x)=\pi-\Lambda(x)$ for each $x\in \mathds{R}$. Then it follows for each $i\in \{i,j,k\}$
\begin{equation}\label{formula-tilde-xita-euclid}
\tilde{\theta}_i(x_1,x_2,x_3)=\Lambda\bigg(\frac{x_j^2+x_{k}^2-x_{i}^2}{2x_{j}x_{k}}\bigg)
\end{equation}
in Euclidean background geometry and
\begin{equation}\label{formula-tilde-xita-hyper}
\tilde{\theta}_i(x_1,x_2,x_3)=\Lambda\bigg(\frac{\cosh x_j\cosh x_{k}-\cosh x_{i}}{\sinh x_{j}\sinh x_{k}}\bigg)
\end{equation}
in hyperbolic background geometry. Note (\ref{formula-tilde-xita-euclid}) and (\ref{formula-tilde-xita-hyper}) also implies that $\tilde{\theta}_i\in C(\mathds{R}^3_{>0})$.

For a compact surface $M$ with triangulation $\mathcal{T}$, let $d:E\rightarrow(0,+\infty)$ be a positive function assigning each edge $\{ij\}\in E$ a length $d_{ij}$. We call $d$ a hyperbolic (or Euclidean) PL-metric (piecewise linear
metric) if for each triangle $\{ijk\}\in F$ (with lengthes $d_{ij}$, $d_{jk}$ and $d_{ik}$), $\{ijk\}$ is isometric to a hyperbolic (or Euclidean) triangle in $\mathds{H}^2$ (or $\mathds{E}^2$). Thus a triangulated surface $(M, \mathcal{T})$ with a hyperbolic (or Euclidean) PL-metric can be obtained by gluing some hyperbolic (or Euclidean) triangles coherently together along edges. Note the discrete Gaussian curvatures (\ref{classical Gauss curv}) is still meaningful for surface with PL-metric. However, if we change hyperbolic (or Euclidean) triangles to generalized hyperbolic (or Euclidean) triangles, and``gluing" the generalized hyperbolic (or Euclidean) triangles together, we can get a generalized hyperbolic (or Euclidean) PL-metric, and then we naturally extend the definition of discrete Gaussian curvature continuously to
\begin{equation}\label{def-K-tuta}
\widetilde{K}_i(d)=2\pi-\sum_{\{ijk\} \in F}\tilde{\theta}_i^{jk}.
\end{equation}
Note each point $d\in\mathds{R}^{|E|}_{>0}$ is in fact a generalized hyperbolic (or Euclidean) PL-metric, hence $\widetilde{K}_i$ is a continuous function defined on $\mathds{R}^{|E|}_{>0}$. For the extended curvature $\widetilde{K}$, we still have a discrete version of Gauss-Bonnet formula
\begin{equation}\label{Gauss-Bonnet-extend}
\sum_{i=1}^N\widetilde{K}_i=2\pi \chi(M)+\lambda Area(M).
\end{equation}
\subsection{Luo's extension method for locally convex function}
Inversive distance circle pattern is a way to produce generalized hyperbolic (or Euclidean) PL-metrics. Given a triangulated surface $(M, \mathcal{T})$ with inversive distance $I\geq0$ in hyperbolic (or Euclidean) background geometry. Consider a generic triangle $\{ijk\}\in F$, which is configured by three circles with inversive distance $I_{ij}$, $I_{jk}$ and $I_{ik}$. With out loss of generality, we may suppose the three vertices $i$, $j$, $k$ appear in this order in the ordered vertex sequence $1,\cdots,N$. Denote
\begin{equation}
\Delta_{ijk}\triangleq\left\{(r_i,r_j,r_k)\in\mathds{R}^3_{>0}\,\big|\,(l_{ij},l_{jk},l_{ik})\in \Xi\right\},
\end{equation}
where $l_{ij}$ is determined by (\ref{formula-inver-dist-hyper}) (or (\ref{Def-edge-length})) in hyperbolic (or Euclidean) background geometry. In the following, we only consider the hyperbolic background geometry. We use $\ln\tanh\frac{\Delta_{ijk}}{2}$ as the homeomorphic image of
$\Delta_{ijk}$ under coordinate change $u_i=\ln\tanh\frac{r_i}{2}$. Guo proved (see Lemma 11-12, Corollary 13, \cite{Guoren})
\begin{lemma}\label{lemma-guoren-define-W}
In hyperbolic background geometry, $\Delta_{ijk}$ is a non-empty simply connected open subset of $\mathds{R}^3_{>0}$. The Jacobian matrix of functions $\theta_i^{jk}$, $\theta_j^{ik}$, $\theta_k^{ij}$ in terms of $u_1$, $u_2$, $u_3$ is symmetric and negative definite.
For any fixed $c\in\ln\tanh\frac{\Delta_{ijk}}{2}$, the integration $W_{ijk}(u_i,u_j,u_k)\triangleq\int_{c}^{(u_i,u_j,u_k)}(\theta_i^{jk}du_i+\theta_j^{ik}du_j+\theta_k^{ij}du_k)$ is well defined. It is a locally strictly concave function on $\ln\tanh\frac{\Delta_{ijk}}{2}$ and satisfies $\frac{\partial W_{ijk}}{\partial u_i}=\theta_i^{jk}$.
\end{lemma}
Using the locally concave property of $W_{ijk}$, Guo proved local rigidity for inversive distance circle patterns. To obtain global rigidity, one needs to extend $W_{ijk}$ to a global concave function. \cite{Bobenko} first provided a method to do so (in discrete conformal change setting), while Luo \cite{Luo1} generalized \cite{Bobenko}'s work to more generic $1$-forms and locally convex functions. A differential $1$-form $\omega=\sum_{i=1}^na_i(x)dx_i$ in an open set $U\subset\mathds{R}^n$ is said to be continuous if each $a_i(x)$ is a continuous function on $U$. A continuous $1$-form $\omega$ is called closed if $\int_{\partial\tau}\omega=0$ for any Euclidean triangle $\tau\subset U$. By the standard approximation theory, if $\omega$ is closed and $\gamma$ is a piecewise $C^1$-smooth null homologous loop in $U$, then $\int_{\gamma}\omega=0$. If $U$ is simply connected, then in the integral
$G(x)=\int_{a}^x\omega$ is well defined (where $a\in U$ is arbitrary chose), independent of the choice of piecewise smooth paths in $U$ from $a$ to $x$. Moreover, the function $G(x)$ is $C^1$-smooth so that $\frac{\partial G(x)}{\partial x_i}=a_i(x)$. Luo proved
\begin{lemma}[$C^{1,1}$-extend]\label{lemma-luo-essential}
The $C^{\infty}$-smooth $1$-form $\theta_i^{jk}du_i+\theta_j^{ik}du_j+\theta_k^{ij}du_k$, originally defined on $\ln\tanh\frac{\Delta_{ijk}}{2}$,
 can be extended to a continuous closed $1$-form $\tilde{\theta}_i^{jk}du_i+\tilde{\theta}_j^{ik}du_j+\tilde{\theta_k}^{ij}du_k$, defined on $\mathds{R}^3_{<0}$, so that the integration
\begin{equation}\label{def-F-tuta-123}
\widetilde{W}_{ijk}(u_i,u_j,u_k)\triangleq\int_{c}^{(u_i,u_j,u_k)} \tilde{\theta}_i^{jk} du_i+\tilde{\theta}_j^{ik} du_j+\tilde{\theta}_k^{ij} du_k, \,\,\,(u_i,u_j,u_k)\in\mathds{R}^3_{<0}
\end{equation}
is a $C^1$-smooth concave function, where $c$ is arbitrary chosen in $\mathds{R}^3_{<0}$.
\end{lemma}
\subsection{The extended flow}
Now we consider the extended combinatorial Ricci flow
\begin{equation}\label{Def-hyper-flow-u-extend}
\begin{cases}
{u_i}'(t)=-\widetilde{K}_i\\
\,u(0)\in \ln\tanh\frac{\Omega}{2}
\end{cases}.
\end{equation}

Similar to Lemma \ref{Lemma-edge-big-then-angle-zero}, we have
\begin{lemma}\label{Lemma-edge-big-then-angle-zero-extended}
For any $\epsilon>0$, there exists a number $l$ so that when $r_i>l$, the generalized inner angle $\tilde{\theta}_i$ in the generalized hyperbolic triangle $\triangle v_iv_jv_k$ is smaller than $\epsilon$.
\end{lemma}
\begin{proof}
If $r_i$ is large enough, then $l_{ij}+l_{ik}>l_{jk}$. If further $l_{ij}+l_{jk}\leq l_{ik}$ or $l_{ik}+l_{jk}\leq l_{ij}$, then obviously $\tilde{\theta}_i=0$. Else, we have $\tilde{\theta}_i=\theta_i$, and in this case, using Lemma \ref{Lemma-edge-big-then-angle-zero} we get the conclusion.
\end{proof}

\begin{theorem}[Extended flow]\label{thm-long-term-existence-extended}
The solution $u(t)$ to the extended flow (\ref{Def-hyper-flow-u-extend}) exists for all time $t\geq 0$.
\end{theorem}
\begin{proof}
On one hand, note all the extended curvatures $\widetilde{K}_i$ are uniformly bounded too. Then the proof in Proposition \ref{Prop-positive-low-bound} is still valid, and hence $r_i(t)>0$ on any finite time interval. On the other hand, by Lemma \ref{Lemma-edge-big-then-angle-zero-extended}, the proof in Proposition \ref{Prop-uniform-up-bound} can be moved here almost word for word. Hence all $r_i(t)$ are uniformly bounded above. Thus we get the conclusion above.
\end{proof}

\subsection{Combinatorial Ricci potential}
Now fix an arbitrary point $u_0\in \ln\tanh\frac{\Omega}{2}$. Assume $W_{ijk}(u_i,u_j,u_k)$ is defined as in Lemma \ref{lemma-guoren-define-W} and $\widetilde{W}_{ijk}(u_i,u_j,u_k)$ is defined as in (\ref{def-F-tuta-123}) with the same $c=(u_{0,i},u_{0,j},u_{0,k})$. For any $u\in \ln\tanh\frac{\Omega}{2}$, Guo \cite{Guoren} and Luo \cite{Luo1} introduced the following functional
\begin{equation}
F(u)\triangleq 2\pi\sum_{i=1}^N(u_i-u_{0,i})-\sum_{\{ijk\}\in F}W_{ijk}(u_i,u_j,u_k).
\end{equation}
We call above functional combinatorial Ricci potential in hyperbolic background geometry. By Lemma \ref{lemma-luo-essential}, the combinatorial Ricci potential, originally defined on $\ln\tanh\frac{\Omega}{2}$, can be extended $C^1$-smoothly to
\begin{equation}
\widetilde{F}(u)\triangleq 2\pi\sum_{i=1}^N(u_i-u_{0,i})-\sum_{\{ijk\}\in F}\widetilde{W}_{ijk}(u_i,u_j,u_k),
\end{equation}
which is defined for all $u\in\mathds{R}^N_{<0}$. It's easy to show that $\nabla_u F=K$ and $\nabla_u \widetilde{F}=\widetilde{K}$. In \cite{Luo1}, Luo proved
\begin{lemma}\label{prop-hessF}
$\widetilde{F}(u)\in C^1(\mathds{R}^N_{<0})\cap C^{\infty}(\ln\tanh\frac{\Omega}{2})$, and $\widetilde{F}(u)$ is convex on $\mathds{R}^N_{<0}$. Moreover, $hess\widetilde{F}$ is positive definite on $\ln\tanh\frac{\Omega}{2}$.
\end{lemma}
We can't expect higher order smoothness of $\widetilde{F}$. Actually, by calculation one can easily find $\partial \theta_i^{jk}/\partial r_j=\infty$ at the ``triangle inequality invalid boundary", hence $\widetilde{K}_i$ is not $C^1$-smooth at the ``triangle inequality invalid boundary". It follows that $\widetilde{F}$ is not $C^2$-smooth if the ``triangle inequality invalid boundary" exists, or equivalently, $I_{ij}>1$ for at least one edge $i\thicksim j$.
\begin{proposition}\label{Prop-F-tura-infinity}
If there exists a zero curvature metric $r^*\in\Omega$, then $\lim\limits_{u\rightarrow\infty; \,u\in \mathds{R}^N_{<0}}\widetilde{F}(u)=+\infty.$
\end{proposition}
\begin{proof}
Let $u^*\in\ln\tanh\frac{\Omega}{2}$ be the corresponding $u$-coordinate of $r^*$. Choose $\delta>0$, $\delta<dist(u^*, \partial\ln\tanh\frac{\Omega}{2})$. For each direction $\xi\in \mathbb{S}^{N-1}$, the ray $\overrightarrow{u^*+t\xi}$, $t\geq0$ intersect with $\partial\mathds{R}^N_{<0}$ (consider ``$\infty$" as one boundary point) at some point $u^*+a_{\xi}\xi$. Then $\varphi_{\xi}(t)=\widetilde{F}(u^*+t\xi)$ is well defined on the time interval $t\in [0,a_{\xi})$. It's easy to see $[0,\delta]\subset [0,a_{\xi})$. Since $\widetilde{F}$ is convex, then $\varphi_{\xi}(t)$ is convex on $[0,a_{\xi})$, hence $\varphi_{\xi}'(t)$ is non-decreasing on $[0,a_{\xi})$. Moreover, for $t\in[0,\delta]$, $\varphi_{\xi}''(t)=\xi^T hessF \xi>0$, this shows that $\varphi_{\xi}'(t)$ is strictly increasing on $[0,\delta]$. Note $\varphi_{\xi}'(0)=0$, then it follows $\varphi_{\xi}'(t)\geq\varphi_{\xi}'(\delta)>0$ for all $t\in[\delta,a_{\xi})$, which implies that $\varphi_{\xi}(t)$ is strictly monotone increasing on $t\in[0,a_{\xi})$. For $t\geq \delta$, we have $\varphi_{\xi}(t)\geq\varphi_{\xi}(\delta)+\varphi'_{\xi}(\delta)(t-\delta)$. If $a_{\xi}=+\infty$, let $t\rightarrow a_{\xi}$, then it follows $\varphi_{\xi}(t)\rightarrow+\infty$. By the following elementary Lemma \ref{lemma-gotoinfinity}, we get the conclusion above.
\end{proof}
\begin{lemma}\label{lemma-gotoinfinity}
Let $f\in C(\mathds{R}^n_{<0})$. Suppose there is a $u^*\in\mathds{R}^n_{<0}$, such that for any direction $\xi\in \mathbb{S}^{n-1}$, $f(u^*+t\xi)$ is monotone increasing for $t\in[0,a_{\xi})$, where $a_{\xi}$ is the time that the ray $\overrightarrow{u^*+t\xi}$ intersects with $\partial\mathds{R}^N_{<0}$, furthermore, $f(u^*+t\xi)$ tends to $+\infty$ as $t\rightarrow a_{\xi}$ whenever $a_{\xi}=+\infty$. Then $\lim\limits_{x\rightarrow\infty;\,x\in\mathds{R}^n_{<0}}f(x)=+\infty.$
\end{lemma}
Since the proof of above lemma is similar with Lemma B.1 in \cite{Ge-Xu1} and Lemma 3.10 in \cite{Ge-Jiang1}. We omit the details here.

\begin{theorem}\label{thm-unique-zero-curvature}
Assume there exists a zero curvature metric $r^*\in\Omega$. Then it is unique in $\mathds{R}^N_{>0}$.
\end{theorem}
\begin{proof}
Note $\varphi_{\xi}'(t)=\xi^T\widetilde{K}$. Whenever there is a zero curvature metric $\hat{u}$ in $\mathds{R}^N_{<0}$, there corresponds a $\hat{t}\in[0,a_{\xi})$ with $\varphi_{\xi}'(\hat{t})=0$. However, $\varphi_{\xi}'(t)$ has the unique zero point at $t=0$ from the proof above. Hence the zero curvature metric is unique.
\end{proof}
\subsection{Proof of Theorem \ref{Thm-converge-introduction}} We just need to prove the following convergence result for the extended flow.
\begin{theorem}[Long time convergence]\label{Thm-converge-zhengwen}
Given $(M, \mathcal{T})$ with inversive distance $I\geq 0$ in hyperbolic background geometry. If there exists a metric of zero curvature $u^*\in\ln\tanh\frac{\Omega}{2}$, then the solution $u(t)$ to the extended flow (\ref{Def-hyper-flow-u-extend}) converges exponentially fast to $u^*$ as $t\rightarrow+\infty$.
\end{theorem}
\begin{proof}
By Proposition \ref{Prop-F-tura-infinity}, $\widetilde{F}(u)$ tends to $+\infty$ as $\|u\|$ tends to $+\infty$. It's easy to see $\frac{d}{dt}\widetilde{F}(u(t))=-\|\widetilde{K}\|^2\leq 0$ along flow (\ref{Def-hyper-flow-u-extend}) and then $\widetilde{F}(u(t))$ is decreasing. Hence $\|u(t)\|$ is uniformly bounded from above. This implies that all $r_i(t)$ are uniformly bounded from below from a positive constant. By Proposition \ref{Prop-uniform-up-bound}, all $r_i(t)$ are uniformly bounded from above. Thus the solution $\{r(t)\}$ lies in a compact region in $\mathds{R}^N_{>0}$. Then it's easy to show that there exists a sequence $t_n\uparrow+\infty$ so that $u(t_n)\rightarrow u^*$. By Theorem \ref{Thm-0curv-metric-exist-imply converg}, $u^*$ is the unique asymptotically stable point of the flow (\ref{Def-hyper-flow-u-extend}), and for some sufficient big $t_{n_0}$, the solution $\{u(t)\}_{t\ge t_{n_0}}$ converges exponentially fast to $u^*$, i.e., the original flow $\{u(t)\}_{t\ge 0}$ converges exponentially fast to the zero curvature metric $u^\ast$.
\end{proof}
\begin{remark}
Theorem \ref{thm-long-term-existence-extended} and Theorem \ref{Thm-converge-zhengwen} are still true, for arbitrary initial value $u(0)\in\mathds{R}^N_{<0}$ in (\ref{Def-hyper-flow-u-extend}).
\end{remark}
\subsection{Deform inversive distance circle pattern to prescribed curvature} We can deform any inversive distance circle pattern to a pattern with prescribed curvatures if it is admissible.
\begin{definition}
Denote $K(\Omega)\triangleq\big\{K(r)|r\in \Omega\big\}$. Each prescribed $\bar{K}$ with $\bar{K}\in K(\Omega)$ is called admissible. If $\bar{r}\in\Omega$ such that $\bar{K}=K(\bar{r})$, we say $\bar{K}$ is realized by $\bar{r}$.
\end{definition}

Given a triangulated surface $(M, \mathcal{T})$ with inversive distance $I\geq 0$ in hyperbolic background geometry. Let $\bar{K}\in\mathds{R}^N$ be any prescribed curvature, consider the prescribed flow
\begin{equation}\label{Def-hyper-flow-u-prescribed}
\begin{cases}
{u_i}'(t)=\bar{K}_i-K_i\\
\,u(0)\,\in\ln\tanh\frac{\Omega}{2}
\end{cases},
\end{equation}
then Proposition \ref{Prop-maxm-exists-time} and Proposition \ref{Prop-positive-low-bound} are true. If further assume $\bar{K}_i<2\pi$ for all $i\in V$, then Proposition \ref{Prop-uniform-up-bound} is also true. Note that to get a uniform upper bound for $r_i(t)$, some assumptions on $\bar{K}_i$, such like $\bar{K}_i<2\pi$, is necessary. If all $\bar{K}_i>2\pi$, then $r_i'(t)=(\bar{K}_i-K_i)\sinh r_i>c>0$, hence $r_i(t)$ is strictly increasing, and can't uniformly bounded from above. 


Theorem \ref{Thm-0curv-metric-exist-imply converg} now can be generalized to
\begin{theorem}\label{Thm-0curv-metric-exist-imply-converg-prescribed}
Given $(M, \mathcal{T})$ with inversive distance $I\geq 0$ in hyperbolic background geometry.
\begin{enumerate}
  \item If the solution $\{r(t)\}\subset\Omega$ to flow (\ref{Def-hyper-flow-u-prescribed}) converges to a metric $\bar{r}$, then $\bar{r}$ realized $\bar{K}$, which implies that $\bar{K}$ is admissible.
  \item Conversely, assume $\bar{K}$ is admissible. Then the solution $r(t)$ to flow (\ref{Def-hyper-flow-u-prescribed}) exists for all time $t\geq0$ and converges exponentially fast to $\bar{r}$ if $r(0)$ is close enough to $\bar{r}$.
\end{enumerate}
\end{theorem}

By a similar analysis into the following prescribed discrete Ricci potential
\begin{equation}
\widetilde{F}_p(u)\triangleq \sum_{i=1}^N(2\pi-\bar{K}_i)(u_i-u_{0,i})-\sum_{\{ijk\}\in F}\widetilde{W}_{ijk}(u_i,u_j,u_k),
\end{equation}
we can generalize Theorem \ref{thm-long-term-existence-extended}, Theorem \ref{thm-unique-zero-curvature} and Theorem \ref{Thm-converge-zhengwen} to the following three theorems.

\begin{theorem}
Given $(M, \mathcal{T})$ with inversive distance $I\geq 0$ in hyperbolic background geometry. Let $\bar{K}\in\mathds{R}^N $ be any prescribed curvature, consider the extended prescribed flow
\begin{equation}\label{Def-hyper-flow-u-extended-prescribed}
\begin{cases}
{u_i}'(t)=\bar{K}_i-\widetilde{K}_i\\
\,u(0)\,\in\,\mathds{R}^N_{<0}
\end{cases}.
\end{equation}
If $\bar{K}_i<2\pi$ for all $i\in V$, then every solution $u(t)$ to (\ref{Def-hyper-flow-u-extended-prescribed}) exists for all time $t\geq 0$.
\end{theorem}

\begin{theorem}\label{thm-unique-admissible-curvature}
Any curvature $\bar{K}\in K(\Omega)$ is realized by an unique metric $\bar{r}$ in the extended space $\mathds{R}^N_{>0}$.
\end{theorem}

\begin{theorem}
Given $(M, \mathcal{T})$ with inversive distance $I\geq 0$ in hyperbolic background geometry. Assume $u(t)$ is a solution to (\ref{Def-hyper-flow-u-extended-prescribed}). If $\bar{K}$ is admissible, then $u(t)$ converges exponentially fast to some $\bar{u}\in\ln\tanh\frac{\Omega}{2}$ as $t\rightarrow+\infty$, and $\bar{u}$ is the unique metric that realized $\bar{K}$. Conversely, if
$u(t)$ converges to some $\bar{u}\in\ln\tanh\frac{\Omega}{2}$, then $\bar{K}$ is realized by $\bar{u}$ and hence is admissible.
\end{theorem}
Since all these results are similar with previous sections, we omit their proofs here.

\section{Degeneration of inversive distance circle patterns}\label{section-degenerate-pattern}
We study $K(\Omega)$ in this section. If the inversive distance $I\in[0,1]$, which is equivalent to say Andreev-Thurston's circle pattern with weight $\Phi\in[0,\frac{\pi}{2}]$, the classical Andreev-Thurston's theorem (see \cite{Andreev2, Andreev1, T1, Marden-Rodin, Colindev} for details) described the shape of $K(\Omega)$ completely. They show that all admissible curvatures form an open convex polytope in $\mathds{R}^N$. A nice expression of Andreev-Thurston's theorem can be found in Theorem 1 of \cite{Guoren}. For more general inversive distance $I\geq0$, the last paragraph of \cite{Ge-Jiang1} shows that there are intrinsic difficulties, caused by non-coherence of ``triangle inequality invalid boundary", to describe $K(\Omega)$ completely. We give a little results here, which generalize Andreev-Thurston's theorem to some extent.

\subsection{Combinatorial and topological obstructions} For any nonempty proper subset $A\subset V$, let $F_A$ be the subcomplex whose vertices are in $A$ and let $Lk(A)$ be the set of pairs $(e, v)$ of an edge $e$ and a vertex $v$ satisfying the following three conditioins: (1) The end points of $e$ are not in $A$; (2) $v$ is in $A$; (3) $e$ and $v$ form a triangle. For any nonempty proper subset $A\subset V$, in \cite{Ge-Jiang1} we had defined
\begin{equation}\label{def-Y-A}
Y_A\triangleq\Big\{x\in \mathds{R}^N \Big|\sum_{i\in A}x_i >-\sum_{(e,v)\in Lk(A)}\big(\pi-\Lambda(I_e)\big)+2\pi\chi(F_A)\Big\},
\end{equation}
where $\Lambda$ is the auxiliary function (\ref{def-big-lamda}). Similar to Theorem 5.4 in \cite{Ge-Jiang1}, we have
\begin{theorem}\label{Thm-2}
Given $(M, \mathcal{T})$ with inversive distance $I\geq 0$ in hyperbolic background geometry. For any nonempty proper subset $A\subset V$,
$\sum\limits_{i\in A}K_i >-\sum\limits_{(e,v)\in Lk(A)}\big(\pi-\Lambda(I_e)\big)+2\pi\chi(F_A).$
\end{theorem}

\begin{proof}
We first prove in a single triangle $\{ijk\}\in F$, for all $(r_i,r_j,r_k)\in\Delta_{ijk}$, $0<\theta_i(r_i,r_j,r_k)<\pi-\Lambda(I_{jk})$.
For this, we just need to prove for any fixed $\bar{r}_j, \bar{r}_k\in (0,+\infty)$, $\theta_i(r_i, \bar{r}_j, \bar{r}_k)<\pi-\Lambda(I_{jk})$. Let
\begin{equation*}
\begin{aligned}
f(r_i)=\cosh^{-1}(\cosh r_i\cosh \bar{r}_j+I_{ij}\sinh r_i\sinh\bar{r}_j)+&\cosh^{-1}(\cosh r_i\cosh \bar{r}_k+I_{ik}\sinh r_i\sinh\bar{r}_k)\\
-&\cosh^{-1}(\cosh \bar{r}_j\cosh \bar{r}_k+I_{jk}\sinh \bar{r}_j\sinh\bar{r}_k)
\end{aligned}
\end{equation*}
It's easy to see $f(+\infty)=+\infty$ and $f'(r_i)>0$. Moreover, if $I_{jk}>1$, then $f(0)<0$. If $I_{jk}=1$, then $f(0)=0$. While if $0\leq I_{jk}<1$, then $f(0)>0$. Obviously, if $I_{jk}>1$, then equation $f(r_i)=0$ has an unique positive solution $\bar{r}_i$. If $I_{jk}\in[0,1]$, set $\bar{r}_i=0$. On one hand, by the law of cosines in hyperbolic geometry,
$$\cos \theta_i=\frac{\cosh l_{ij}\cosh l_{ik}-\cosh l_{jk}}{\sinh l_{ij}\sinh l_{ik}},$$
and taking limit, by careful calculation, we get $\lim\limits_{r_i\rightarrow \bar{r}_i}\theta_i(r_i, \bar{r}_j, \bar{r}_k)=\pi-\Lambda(I_{jk})$ and $\lim\limits_{r_i\rightarrow +\infty}\theta_i(r_i, \bar{r}_j, \bar{r}_k)=0$.
On the other hand, by Guo's Lemma \ref{lemma-guoren-define-W}, $r_i\partial \theta_i/\partial r_i=\partial \theta_i/\partial u_i<0$, implying that $\theta_i$ is a strictly decreasing function of $r_i$. Hence $0<\theta_i(r_i, \bar{r}_j, \bar{r}_k)<\pi-\Lambda(I_{jk})$, and then for all $(r_i,r_j,r_k)\in\Delta_{ijk}$, we get $0<\theta_i<\pi-\Lambda(I_{jk})$.

Next we follow the approach pioneered by Marden and Rodin \cite{Marden-Rodin} to finish the proof. Consider all the triangles in $F$ having a vertex in $A$. These triangles can be classified into three types $A_1$, $A_2$ and $A_3$. For each $i\in\{1,2,3\}$, a triangle is in $A_i$ if and only if it has exactly $i$ many vertices in $A$. By what is proved above, $\sum\limits_{i\in A, \{ijk\}\in A_1} \theta_i^{jk}<\sum\limits_{(e,v)\in Lk(A)}\big(\pi-\Lambda(I_e)\big)$. Also note $\theta_i^{jk}+\theta_j^{ik}<\theta_i^{jk}+\theta_j^{ik}+\theta_k^{ij}<\pi$, then it follows
\begin{equation*}
\begin{aligned}
\sum_{i\in A}K_i=&\sum_{i\in A}\Big(2\pi-\sum_{\{ijk\} \in F}\theta_i^{jk}\Big)=2\pi|A|-\sum_{i\in A}\sum_{\{ijk\} \in F}\theta_i^{jk}\\
=&2\pi|A|-\Big(\sum_{i\in A, \{ijk\}\in A_1} \theta_i^{jk}+\sum_{i,\;j\in A, \{ijk\}\in A_2} \big(\theta_i^{jk}+\theta_j^{ik}\big)+
\sum_{\{ijk\}\in A_3} \big(\theta_i^{jk}+\theta_j^{ik}+\theta_k^{ij}\big)\Big)\\[6pt]
>&2\pi|A|-\sum_{(e,v)\in Lk(A)}\big(\pi-\Lambda(I_e)\big)-\pi|A_2|-\pi|A_3|\\
=&-\sum_{(e,v)\in Lk(A)}\big(\pi-\Lambda(I_e)\big)+2\pi\Big(|A|-\frac{|A_2|}{2}-\frac{|A_3|}{2}\Big)\\
=&-\sum_{(e,v)\in Lk(A)}\big(\pi-\Lambda(I_e)\big)+2\pi\chi(F_A).
\end{aligned}
\end{equation*}
\end{proof}

\begin{corollary}\label{coroll-curvature-contain-convex-set}
Given $(M, \mathcal{T})$ with inversive distance $I\geq 0$ in hyperbolic background geometry. Then the space of all admissible discrete Gaussian curvatures $K(\Omega)$ is contained in a convex set $\mathop{\bigcap}\limits_{\phi\neq A\subsetneqq V} Y_A$.
\end{corollary}

\begin{corollary}\label{coroll-zero-curvature-comb-topo-condition}
Given $(M, \mathcal{T})$ with inversive distance $I\geq 0$ in hyperbolic background geometry. If there exists a zero curvature metric, then for each nonempty proper subset $A\subset V$, $\sum\limits_{(e,v)\in Lk(A)}\big(\pi-\Lambda(I_e)\big)>2\pi\chi(F_A)$.
\end{corollary}

\begin{lemma}\label{Lemma-limit-xita}
Assume $b, c\in(0, +\infty]$, then in the generalized hyperbolic triangle $\{ijk\}\in F$ which is configured by three circles with non-negative inversive distance $I_{ij}$, $I_{jk}$ and $I_{ik}$,
\begin{equation}
\lim_{(r_i, r_j, r_k)\rightarrow (0,\,b,\,c)}\tilde{\theta}_i(r_i, r_j, r_k)=\pi-\Lambda(I_{jk}).
\end{equation}
\begin{equation}
\lim_{(r_i,r_j,r_k)\rightarrow (0,\,0,\,c)}\tilde{\theta}_k(r_i, r_j, r_k)=0.
\end{equation}
\end{lemma}
\begin{proof} For any $(r_i,r_j,r_k)\in\mathds{R}^3_{>0}$, we have
\begin{equation*}
\tilde{\theta}_i(r_i,r_j,r_k)=\Lambda\bigg(\frac{\cosh l_{ij}\cosh l_{ik}-\cosh l_{jk}}{\sinh l_{ij}\sinh l_{ik}}\bigg).
\end{equation*}
One can get the conclusion by careful calculations. We omit the tedious but elementary details here.
\end{proof}

\begin{proposition}\label{degenerate}
Given $(M, \mathcal{T})$ with inversive distance $I\geq 0$ in hyperbolic background geometry. Assume there is a sequence of
$r^{(n)}=\big(r_1^{(n)},...,r_N^{(n)}\big)^T\in \mathds{R}^N_{>0}$ and a nonempty proper subset $A\subset V$, so that
$\lim\limits_{n\rightarrow+\infty}r_i^{(n)}=0$ for $i\in A$ and $\lim\limits_{n\rightarrow+\infty}r_i^{(n)}>0$ (may be $+\infty$) for $i\notin A$, then
\begin{equation}\label{limit-singular-behavior}
\lim\limits_{n\rightarrow+\infty}\sum_{i\in \,A}\widetilde{K}_i(r^{(n)})=-\sum_{(e,v)\in Lk(A)}(\pi-\Lambda(I_e))+2\pi\chi(F_A).
\end{equation}
\end{proposition}
\begin{proof}
If $\{ijk\}\in A_2$, then by Lemma \ref{Lemma-limit-xita}, $\tilde{\theta}_i^{jk(n)}\rightarrow\pi-\Lambda(I_e)$. If $\{ijk\}\in A_2$, then the generalized hyperbolic triangle degenerates to a geodesic segment and hence $\tilde{\theta}_i^{jk(n)}+\tilde{\theta}_j^{ik(n)}\rightarrow \pi$. If $\{ijk\}\in A_3$, then the generalized hyperbolic triangle is shrinking to a point and hence $\tilde{\theta}_i^{jk(n)}+\tilde{\theta}_j^{ik(n)}+\tilde{\theta}_k^{ij(n)}\rightarrow \pi$ too. Therefore,
\begin{equation*}
\begin{aligned}
\sum_{i\in A}\widetilde{K}_i^{(n)}
=\;&2\pi|A|-\Big(\sum_{i\in A, \{ijk\}\in A_1} \tilde{\theta}_i^{jk(n)}+\sum_{i,\;j\in A, \{ijk\}\in A_2} \big(\tilde{\theta}_i^{jk}+\tilde{\theta}_j^{ik}\big)^{(n)}+
\sum_{\{ijk\}\in A_3} \big(\tilde{\theta}_i^{jk}+\tilde{\theta}_j^{ik}+\tilde{\theta}_k^{ij}\big)^{(n)}\Big)\\[6pt]
\rightarrow&2\pi|A|-\sum_{(e,v)\in Lk(A)}\big(\pi-\Lambda(I_e)\big)-|A_2|\pi-|A_3|\pi=-\sum_{(e,v)\in Lk(A)}\big(\pi-\Lambda(I_e)\big)+2\pi\chi(F_A).
\end{aligned}
\end{equation*}
\end{proof}

\begin{theorem}\label{com-inequi-1}
Given $(M, \mathcal{T})$ with inversive distance $I\geq 0$ in hyperbolic background geometry. Then the space of all possible extended curvatures $\widetilde{K}(\mathds{R}^N_{>0})\subset\mathop{\bigcap}\limits_{\phi\neq A\subsetneqq V} \overline{Y_A}$. ($\overline{Y_A}$ is the closer of $Y_A$)
\end{theorem}
\begin{proof}
We need to prove that for every $r\in \mathds{R}^N_{>0}$, the extended curvature $\widetilde{K}$ satisfies
\begin{equation*}
\sum_{i\in \,A}\widetilde{K}_i(r)\geq-\sum_{(e,v)\in Lk(A)}(\pi-\Lambda(I_e))+2\pi\chi(F_A)
\end{equation*}
for each nonempty proper subset $A\subset V$. For this, we just need to prove that in a single triangle $\{ijk\}\in F$, for all $(r_i,r_j,r_k)\in\mathds{R}^3_{>0}$, $0\leq\tilde{\theta}_i\leq\pi-\Lambda(I_{jk})$. Since the methods are almost the same with proof of Theorem 5.8 in \cite{Ge-Jiang1}, we omit the details here.
\end{proof}

Consider a triangle $\{ijk\}\in F$ that is configured by three circles with three fixed non-negative numbers $I_{ij}$, $I_{jk}$ and $I_{ik}$ as inversive distances. Let $\theta_i$, $\theta_j$ and $\theta_k$ be the three inner angles. For this (a single triangle) case, similar to Theorem 5.9 in \cite{Ge-Jiang1}, we can determine the shape of $K(\Omega)$ (in this case $\Omega=\Delta_{ijk}$) completely.
\begin{theorem}\label{thm-angle-range-one-triangle}
$(\theta_i, \theta_j, \theta_k)$ is a diffeomorphism from $\Delta_{ijk}$ to $Z$, where
\begin{equation*}\label{angle-range-asingletriangle}
Z=\left\{(\theta_i, \theta_j, \theta_k)\in\mathds{R}^3\,\big|\,\theta_i+\theta_j+\theta_k<\pi; \,0<\theta_{i'}<\pi-\Lambda(I_{j'k'}),\,\forall\, \{i',j',k'\}=\{i,j,k\}\right\}.
\end{equation*}
\end{theorem}

\subsection{Maximum principle}\label{section-max-principle} The maximum principle is very useful in geometric analysis and other fields. Its discrete version was first applied in combinatorial curvature flows by Chow-Luo \cite{CL1} and Glickenstein \cite{Glickenstein2}. Ge-Xu \cite{Ge-Xu3,Ge-Xu4} systematically developed the theory of discrete maximum principle and used it to study the combinatorial Yamabe problem, that is, the existence of constant curvature metrics. Using this, Ge-Xu proved that there exists a zero (or constant) curvature metric if and only if there exists a metric with non-negative curvatures in the hyperbolic (or Euclidean) background geometry. Note that by Andreev, Thurston and other pioneer's work for classical $K$-curvature \cite{Andreev2,Andreev1,T1,Marden-Rodin,Colindev,Ri,Hezhengxu1,Bobenko}, and by Ge-Xu's recent work for newly defined $R$-curvature \cite{Ge-Xu3,Ge-Xu1,Ge-Xu4}, the existence of constant ($K$- or $R$-) curvature metrics, or non-negative curvature metrics are determined by the topological and combinatorial structures of $(M,\mathcal{T})$. It seems amazing that discrete maximum principle can be used to extract topological information of $M$ and combinatorial information of the triangulation $\mathcal{T}$.

Let the background geometry is hyperbolic, we consider Chow and Luo's combinatorial Ricci flow, which is equivalent to require the inversive distance $0\leq I \leq 1$. In this case, $\Omega=\mathds{R}^N_{>0}$, and the curvature $K_i(t)$ evolves according to
\begin{equation}\label{formula-k-evolve}
\frac{dK_i}{dt}=\sum_{j\thicksim i}C_{ij}(K_j-K_i)-B_iK_i,
\end{equation}
where both $C_{ij}$ and $B_i$ are positive. Set $M(t)=max\big\{K_1(t),\cdots,K_N(t),0\big\}$, $m(t)=min\big\{K_1(t),\cdots,K_N(t),0\big\}$. Using the maximum principle, Chow-Luo (corollary 3.3, \cite{CL1}) proved that $M(t)$ is non-increasing while $m(t)$ is non-decreasing. We can show even more.
\begin{theorem}\label{thm-negative-imply-zero}
Given $(M,\mathcal{T})$ with inversive distance $0\leq I\leq1$ in hyperbolic background
geometry. If there exists a metric with non-positive curvatures, then there exist a zero curvature metric.
\end{theorem}
\begin{proof}
This theorem was first proved by Ge-Xu in \cite{Ge-Xu4}, where they defined a new $R$-curvature and a new hyperbolic curvature flow. Using a discrete version of maximum principle for $R$-curvature evolutions along the newly defined hyperbolic curvature flow, they finally proved above conclusion. The proof in \cite{Ge-Xu4} is very complicated, and we need to provide a more direct proof here. We still follow \cite{Ge-Xu4}. We deform the metric $r(t)$ according to Chow-Luo's flow $dr_i/dt=-r_iK_i$, beginning from a initial metric $r(0)$ which is exactly the special metric with non-positive curvatures. By Chow-Luo's results, $M(t)\leq M(0)\leq 0$, and hence all $K_i(t)\leq0$. Thus $dr_i/dt\geq 0$ and every $r_i(t)$ is increasing, which implies that all $r_i(t)$ are uniformly bounded below from a positive constant. By Proposition \ref{Prop-uniform-up-bound}, or by Corollary 3.6 in \cite {CL1}, all $r_i(t)$ are uniformly bounded from above. Thus the solution $\{r(t)\}$ lies in a compact region in $\mathds{R}^N_{>0}$. Using Proposition 3.7 in \cite{CL1}, we get the conclusion above.
\end{proof}

The inversive distance $I\in[0,1]$ is equivalent to Andreev-Thurston's circle pattern with weight $\Phi\in[0,\frac{\pi}{2}]$. In this case, the existence of zero curvature metric is purely combinatorial and topological. In fact, Thurston \cite{T1} proved that there exists a metric with zero curvature if and only if the following two combinatorial conditions are satisfied simultaneously: (1) for any three edges $e_1, e_2, e_3$ forming a null homotopic loop in $M$, if $\sum_{i=1}^3\Phi(e_i)\geq\pi$, then $e_1, e_2, e_3$ form the boundary of a triangle in $F$; (2) for any four edges $e_1, e_2, e_3, e_4$ forming a null homotopic loop in $M$, if $\sum_{i=1}^4\Phi(e_i)\geq2\pi$, then $e_1, e_2, e_3, e_4$ form the boundary of the union of two adjacent triangles. The method to prove Proposition \ref{Prop-uniform-up-bound} is valid for all $I\geq 0$, while the maximum principle method to prove $M(t)$ decreasing and Corollary 3.6 in \cite {CL1} is only valid for $0\leq I\leq1$. By an observation of Guo \cite{Guoren}, $C_{ij}=\frac{\partial K_i}{\partial u_j}$ may be negative for general inversive distance $I\geq 0$, hence the maximum principle may not valid any more.

\section{Some questions}\label{section-open-question}
\subsection{Is $\Omega$ simply connected} If so, then we can define the discrete Ricci potential directly as the line integral
\begin{equation}
G(u)\triangleq\int_{c}^u \sum_{i=1}^NK_idu_i,
\end{equation}
where $c\in \ln\tanh\frac{\Omega}{2}$ is arbitrary chosen. Since $\ln\tanh\frac{\Omega}{2}$ is simply connected, and $\frac{\partial K_i}{\partial u_j}=\frac{\partial K_j}{\partial u_i}$, above line integral is well defined  and is independent on the choice of piecewise smooth paths in $\ln\tanh\frac{\Omega}{2}$ from $c$ to $u$. Also, we can introduce
\begin{equation}
\widetilde{G}(u)\triangleq\int_{c}^u \sum_{i=1}^N\widetilde{K}_idu_i,
\end{equation}
which is the $C^1$-smooth extension of $G(u)$. It's easy to see, $G(u)$ differs from $F(u)$ by a constant, while $\widetilde{G}(u)$ differs from $\widetilde{F}(u)$ by a constant. For $M=\{ijk\}\in F$ a single triangle case, $\Omega$ is in fact $\Delta_{ijk}$ and is simply connected non-empty open set by Guo's Lemma \ref{lemma-guoren-define-W}. For general case, we don't know if $\Omega$ is simply connected. We don't know if $\Omega$ is connected or non-empty even worse.

\subsection{Uniqueness of solution to the extended flow} Let $u(t)$ be the solution to the extended flow (\ref{Def-hyper-flow-u-extend}), when ever $u(t)$ lies in $\ln\tanh\frac{\Omega}{2}$, it is unique since $\widetilde{K}_i=K_i$ is smooth and then locally Lipschitz continuous in $\ln\tanh\frac{\Omega}{2}$. However, $\widetilde{K}_i$ is not Lipschitz continuous near the boundary of $\ln\tanh\frac{\Omega}{2}$. So we don't know whether $u(t)$ is unique. If all inversive distance $I_{ij}$ are in $[0,1]$, in this case, $\Omega$ equals to $\mathds{R}^N_{>0}$, and the extended flow is in fact Chow-Luo's flow, hence the solution $u(t)$ is unique in $[0,+\infty)$. We hope $u(t)$ is unique for any inverse distance $I\geq 0$.

\noindent \textbf{Acknowledgements}: The research is supported by National Natural Science Foundation of China under grant (No.11501027), and Fundamental Research Funds for the Central Universities (Nos. 2015JBM103, 2014RC028 and 2016JBM071).

\bibliographystyle{plain}

\end{document}